\documentclass[12pt,reqno]{amsart}
\usepackage{amssymb,amsfonts,latexsym,amstext,amsmath,epsfig}
\usepackage[left=3.2cm,top=3.2cm,right=3.2cm,bottom=3.2cm]{geometry}

\newcommand{\df}{\dfrac}

\newtheorem{theorem}{Theorem}
\newtheorem{lemma}[theorem]{Lemma}

\numberwithin{equation}{section}

\newtheorem{corollary}[theorem]{Corollary}

\begin{document}

\title{Independence densities of hypergraphs}

\author{Anthony Bonato}
\address{Department of Mathematics\\
Ryerson University\\
Toronto, ON\\
Canada, M5B 2K3}
\email{abonato@ryerson.ca}

\author{Jason I. Brown}
\address{Department of Mathematics and Statistics\\
Dalhousie University\\
Halifax, NS\\
CANADA B3H 3J5}
\email{brown@mathstat.dal.ca}

\author{Dieter Mitsche}
\address{Department of Mathematics\\
Ryerson University\\
Toronto, ON\\
Canada, M5B 2K3}
\email{dmitsche@ryerson.ca}

\author{Pawe{\l } Pra{\l }at}
\address{Department of Mathematics\\
Ryerson University\\
Toronto, ON\\
Canada, M5B 2K3}
\email{pralat@ryerson.ca}


\thanks{The authors gratefully acknowledge support from NSERC, MPrime, and Ryerson University}
\maketitle

\begin{abstract}
We consider the number of independent sets in hypergraphs, which allows us to define the independence density of countable hypergraphs. Hypergraph independence densities include a broad family of densities over graphs and relational structures, such as $F$-free densities of graphs for a given graph $F.$ In the case of $k$-uniform hypergraphs, we prove that the independence density is always rational. In the case of finite but unbounded hyperedges, we show that the independence density can be any real number in $[0,1].$ Finally, we extend the notion of independence density via independence polynomials.
\end{abstract}

\section{Introduction}

Densities of graphs and hypergraphs are well studied parameters in graph theory, and they can provide a convenient tool for measuring properties of infinite graphs. Examples of graph densities are the upper density~\cite{diestel}, homomorphism density~\cite{lovasz}, Tur\'{a}n density~\cite{katona}, cop density~\cite{bhw} and co-degree density~\cite{mubayi} (see also~\cite{epp,ped,peng1,zhu}).  In~\cite{bbkp} we considered the independence and chromatic densities of graphs, and we focus on a generalization of the former density in this paper. The {\em independence density} of a graph $G$ of order $n$, where $n$ is a positive integer, is the number of independent sets of $G$, divided by $2^n $, the total number of subsets of $V(G)$. For countable graphs, independence densities are defined as the limits of independence densities of chains of finite subgraphs; it was shown in \cite{bbkp} that the limit was independent of the chain used. A key result proved in~\cite{bbkp} was that the independence density of a countable graph is rational.

Independence may be viewed as the property of being $K_{2}$-free; that is, not containing a complete graph of order $2$. A more general density, therefore, may be defined in the following way. For $F$ be a connected finite graph, a graph $G$ is $F$-\emph{free} if $G$ does not contain $F$ as an induced subgraph. Define $F(G)$ to be the number of subsets $S$ in $V(G)$ (including the empty set) so that subgraph induced by $S$ in $G$ is $F$-free. Define the $F$\emph{-free density of }$G$ by $F(G)/2^{n},$ where $n=|V(G)|$. We may extend this notion even further by considering the $\mathcal{F}$-free density, where $\mathcal{F}$ is a (possibly infinite) set of finite graphs. This gives rise to the notion of \emph{acyclic density}, where $\mathcal{F}$ is the set of all cycles, and \emph{bipartite-free density}, where $\mathcal{F}$ is the set of all odd cycles.

All of these extensions fit into the broad context of independence densities of hypergraphs. To be more precise, let $H$ be a finite hypergraph (or \emph{set system}) of order $n$, where $n\ge 1$ is an integer. Hence, $H$ consists of vertex set $V=V(H)$ of cardinality $n$ and some collection $E=E(H)$ of subsets of $V$, called the {\em hyperedge set} of $H$.  We write $H=(V(H),E(H))$, or $H=(V,E)$ if $H$ is clear from context. An {\em independent set} of $H$ is a subset of vertices of $H$ that does not contain a hyperedge of $H$. Just as in the case for graphs, the {\em independence density} of $H$, written $\mathrm{id}(H)$, is defined to be $i(H)/2^{n}$, where $i(H)$ is the number of independent sets of $H$ and $n$ is the number of vertices of $H$.

For example, consider a hypergraph whose vertices are those of some fixed graph $G$ along with $E$ being the set of subsets of $V(G)$ containing a copy of $K_2$. In this case, $i(H)$ is the independence density of $G.$ Further, if we replace $K_2$ by a set of finite graphs $\mathcal{F}$, then define a hypergraph $H_{G,\mathcal{F}}$ on the vertices of $G$ whose edges correspond to subsets of $V(G)$ which induces a subgraph in $\mathcal{F}$ (note that hyperedges for this hypergraph are finite). In this setting the independence density of $H_{G,\mathcal{F}}$ is the $\mathcal{F}$-free density of $G$. We may also naturally extend $\mathcal{F}$-free densities to relational structures. Independence densities of hypergraphs $H=(V,E)$ therefore, include a vast class of $\mathcal{F}$-free densities over structures such as graphs, hypergraphs, directed graphs, or ordered sets.

We extend the definition of independence density of finite hypergraphs in the natural way to countable hypergraphs (which we define to be on a countable set of vertices but with edges of finite cardinality) as limits of densities of chains of finite induced subhypergraphs. We collect some basic properties of independence density in the next section. We prove that independence density does not depend on the chain used in Theorem~\ref{uni}, and we give bounds via the hypergraph matching number in Theorem~\ref{bounds}. We prove that for $k$-uniform hypergraphs, the independence density is always rational in Theorem~\ref{rat}. However, in Section~\ref{unbs} we demonstrate that in hypergraphs with hyperedges of unbounded size, the independence density can be any fixed real number in $[0,1]$. In the final section, we extend the notion of independence density via independence polynomials and examine limiting values for different values of their variable $x$.

All the hypergraphs we consider are countable. A general reference on hypergraphs is~\cite{berge}, with~\cite{voloshin} a more modern reference. See~\cite{diestel} for additional background on graphs. Let $\mathcal{H}_k$ be the set of countable hypergraphs whose hyperedges have cardinality at most $k$, and $\mathcal{H}$ to be the set of all countable hypergraphs. For hypergraphs $H_{1}$ and $H_{2}$ with disjoint vertex sets, let $H_{1}\cup H_{2}$ denote their disjoint union. The notions of subhypergraph, spanning subhypergraph, and induced subhypergraph are defined in an analogous manner to the definitions in graphs. If $H_2$ is a proper subhypergraph of $H_1$, then $H_1 - H_2$ is the subhypergraph induced by $V(H_1) \setminus V(H_2).$ The countably infinite clique is denoted by $K_{\omega}$ and its complement by $\overline{K_{\omega}}.$

We use the notation $\mathbb{N}$ for the natural numbers (including $0$), and $\mathbb{N}^+$ for the set of positive integers. We use logarithms in the natural base.

\section{Independence density of infinite hypergraphs}\label{inf}

We collect some elementary properties of independence densities of hypergraphs that we will use throughout the paper. The proof of the following lemma is straightforward but is included for completeness.

\begin{lemma}\label{lem}
Let $H_{1}$ and $H_{2}$ be finite hypergraphs.
\begin{enumerate}
\item If $H_{1}$ is a spanning subhypergraph of $H_{2}$, then $i(H_{2})\le i(H_{1})$.
\item $i(H_{1} \cup H_{2})=i(H_{1})i(H_{2})$ and $\mathrm{id}(H_{1}\cup H_{2})=\mathrm{id}(H_{1})\mathrm{id}(H_{2})$.
\item If $H_{1}$ is a subhypergraph of $H_{2}$, then $\mathrm{id}(H_{2})\le \mathrm{id}(H_{1})$.
\end{enumerate}
\end{lemma}
\begin{proof}

For (1), independent sets of $H_2$ are also independent sets of $H_1.$ For the first equality of (2), note that the union of independent sets in $H_1$ and $H_2$ is independent in $H_{1} \cup H_{2};$ the second equality holds by definition. For (3), let the number of vertices of $H_{i}$ be $n_{i}$, so that $n_{1}\leq n_{2}$. Then we have that
\begin{eqnarray*}
\mathrm{id}(H_{2}) &=&\frac{i(H_{2})}{2^{n_{2}}} \\
&\leq &\frac{i(H_{1}\cup (H_{2}-H_{1}))}{2^{n_{2}}} \\
&=&\frac{i(H_{1})i(H_{2}-H_{1}))}{2^{n_{1}}2^{n_{2}-n_{1}}} \\
&=&\mathrm{id}(H_{1})\mathrm{id}(H_{2}-H_{1}) \\
&\leq &\mathrm{id}(H_{1}),
\end{eqnarray*}
where the first inequality holds by item (1), and the second equality holds by (2). \end{proof}

Let $\mathcal{C} = (H_n:n\ge 0)$ be a family of subhypergraphs of $H$ with the properties that for every integer $n\ge 0$, $H_n$ is an induced subhypergraph of $H_{n+1},$ and
\begin{equation*}
V(H)=\bigcup\limits_{n=0}^{\infty}V(H_{n}).
\end{equation*}
We write $\lim_{n\rightarrow \infty }H_{n}=H,$ and say that $H$ is the \emph{limit of the chain} $\mathcal{C}.$ Given a chain $\mathcal{C}=(H_n: n\ge 0)$ whose limit is $H$ we define the \emph{independence density of $H$ relative to $\mathcal{C}$} by
$$
\mathrm{id}(H,\mathcal{C})=\lim_{n\rightarrow \infty} \mathrm{id}(H_n).
$$

\begin{theorem}\label{uni} Let $H$ be a countable hypergraph.
\begin{enumerate}
\item For all chains $\mathcal{C}$ with limit $H,$ $\mathrm{id}(H,\mathcal{C})$ exists.
\item Let $\mathcal{C}=(H_{n}: n\geq 0)$ and $\mathcal{C}'=(J_{n}: n\geq 0)$ be two chains with the same limiting hypergraph $H.$ Then $$\mathrm{id}(H,\mathcal{C})=\mathrm{id}(H,\mathcal{C}').$$
\end{enumerate}
\end{theorem}
\begin{proof}
Item (1) follows from Lemma~\ref{lem} (3), as the independence densities of a chain form a non-increasing, bounded sequence in $[0,1].$ For (2), fix $\varepsilon > 0$ a real number. Let $x=\mathrm{id}(H,\mathcal{C})$ and $x'=\mathrm{id}(H,\mathcal{C}').$ There is an $n\ge 0$ such that $\mathrm{id}(H_n) \le x+\varepsilon$. We have that $H_n$ is an induced subhypergraph of $J_k$ for some $k\ge 0$. But then Lemma~\ref{lem} (3) implies that
\[ x' \le \mathrm{id}(J_k) \le \mathrm{id}(H_n) \le x+\varepsilon.\]
By symmetry, $x \le x'+\varepsilon.$ The results follows by letting $\varepsilon$ tending to $0$.
\end{proof}

Given Theorem~\ref{uni}, without loss of generality we drop reference to chains, and simplify refer to the \emph{independence density of }$H,$ written $\mathrm{id}(H).$ For an example, consider the hypergraph $H$ with a single edge of size $k$ and infinitely many isolated vertices. In this case, $\mathrm{id}(H) = 1-1/2^k$.

We give a generalization of independence density that will be useful in the next section. For any two finite, disjoint sets $A,B \subseteq V(H)$, denote by $\rho_{A,B} (H)$ the density of independent sets (in all the subsets of $H$) with the property that these independent sets contain all vertices from $A$ and no vertex from $B$. It is straightforward to see that Theorem~\ref{uni} can be generalized to this definition, and so $\rho_{A,B} (H)$ does not depend on the chain used. Indeed, consider a chain $\mathcal{C}=(H_n: n\ge 0)$ whose limit is $H$ such that $A \cup B \subseteq V(H_0)$. Every independent set in $H_n$ under consideration can be extended in at most $2^{r_n}$ ways, where $r_n = |V(H_{n+1}) \setminus V(H_n)|$ (by considering all possible subsets of $V(H_{n+1}) \setminus V(H_n)$). If every extension forms an independent set, then the density remains the same; otherwise it decreases. This shows that the generalized independence densities of a chain form a non-increasing, bounded sequence. The uniqueness of the limit follows by an argument analogous to the one given in the proof of Theorem~\ref{uni}~(2). Finally, let us note that $\rho_{\emptyset,\emptyset}(H) = \mathrm{id}(H)$ and so it is a natural generalization of independence density.

\bigskip

We now give some bounds on the independence density in terms of the matching number. For a hypergraph $H,$ the \emph{matching number} of $H$, written $\mu(H)$, is the supremum of the cardinalities of pairwise non-intersecting hyperedges in $H.$ In Section~\ref{unbs} (see Theorem~\ref{iii}) we provide an example of a hypergraph with infinite matching number, hyperedges of unbounded size, and with $\mathrm{id}(H)>0$. However, if $\mu(H)$ is infinite and $H$ has hyperedges whose cardinalities are bounded above by a universal constant $k\in \mathbb{N}^+$, then it follows from the next result that $\mathrm{id}(H)=0.$

\begin{theorem}\label{bounds} Suppose that $H$ is a countable hypergraph whose hyperedges have cardinality bounded above by $k>0.$ If $\mu(H)$ is finite, then
$$0 < 2^{-k\mu(H)} \le \mathrm{id}(H) \le \left(1-1/2^k\right)^{\mu(H)}.$$
\end{theorem}

\begin{proof} The upper bound follows by noting that by Lemma~\ref{lem}~(3), $\mathrm{id}(H)$ is maximized when $H$ consists of $\mu(H)$ many disjoint hyperedges of cardinality $k$. As a hypergraph consisting of a single hyperedge of cardinality $k$ has density $1-1/2^k$, the proof of the upper bound follows by Lemma~\ref{lem}~(2).

For the lower bound, fix a maximum matching $M$ of $H.$ Now fix a chain $(H_n: n \in \mathbb{N})$ with limit $H$, where $H_0$ is the subhypergraph induced by $M.$ If $H_n$ has order $N$, then we have that $\mathrm{id}(H_n) \ge 2^{N-k\mu(H)}/2^N$, as the vertices not in $M$ form an independent set. The lower bound follows.
\end{proof}

The lower and upper bounds are equal in the case for $k=1$ in Theorem~\ref{bounds}. For graphs $G$ (that is, the case $k=2$) it was shown in~\cite{bbkp} that $(\mu(G)+1)2^{-2\mu(G)}$ is a lower bound for the independence density. Note that this bound is sharp for $G=K_{2m+1}$, if $m>0$ is an integer. It is an open problem to improve the lower bound in Theorem~\ref{bounds} if $k>2$.

\section{Rationality}

In this section, we consider hypergraphs where hyperedges have bounded cardinality.
For $k>0,$ define
\begin{equation*}
\mathbf{H}_k=\{ x\in [0,1]\text{: for some countable }H\in \mathcal{H}_k, \text{ }\mathrm{id}(H)=x \}.
\end{equation*}
(Recall that $\mathcal{H}_k$ is the set of countable hypergraphs whose hyperedges have cardinality at most $k$.) One of our main results is that the set of densities is rational, a fact not obvious a priori. The next theorem generalizes the result in~\cite{bbkp} which proved that independent densities of graphs are rational.

\begin{theorem}\label{rat} For all $k \ge 1$,
$$\mathbf{H}_k \subseteq \mathbb{Q} \cap [0,1].$$
\end{theorem}
\begin{proof}%[Proof of Theorem~\ref{rat}]
Fix a countable hypergraph $H$ whose hyperedges have at most $k$ vertices. We will show that $\mathrm{id}(H) \in \mathbb{Q} \cap [0,1]$. In the case $k=1,$ we leave it to the reader to verify that $\mathrm{id}(H)$ is either $0$, $1$, or $1/2^m$, for $m\ge 1$ an integer.  We assume in the remainder of the proof, therefore, that $k\ge 2.$

Let us recall that for any two finite, disjoint sets $A,B \subseteq V(H)$, we denote by $\rho_{A,B} (H)$ the density of independent sets with the property that these independent sets contain all vertices from $A$ and no vertex from $B$. For a given $A,B$, and any hyperedge $S$ such that $S \cap B = \emptyset$, the set $S \setminus A$ will be called the \emph{out-set of $S$ relative to $A$ and $B$} (we drop reference to $A$ and $B$ if they are clear from context). Note that for any $A,B$, and any $x \in V(H) \setminus (A \cup B)$, we have that
\begin{equation}\label{eq:incl_excl}
\rho_{A,B}(H) = \rho_{A \cup \{x\},B}(H) + \rho_{A,B\cup \{x\}}(H),
\end{equation}
as any independent set containing all $A$ and no $B$ either contains the vertex $x$, in which case this set is counted by $ \rho_{A \cup \{x\},B}(H)$, or it does not contain it, in which case this set is counted by  $\rho_{A,B\cup \{x\}}(H)$.  Moreover,~(\ref{eq:incl_excl}) can be generalized to obtain that for any finite set $W \subseteq V(H) \setminus (A \cup B)$
\begin{equation}\label{eq:incl_excl2}
\rho_{A,B}(H) = \sum_{C \in 2^W} \rho_{A \cup C, B \cup (W \setminus C)}(H).
\end{equation}

We introduce the following useful notation. For given $H$, $A$ and $B$, let $R(H,A,B)$ be the size of the largest out-set (with respect to $H$, $A$ and $B$); observe that $R(H,A,B)$ is bounded above by $k$. We will use the notation $\rho_{A,B}^r (H)$  instead of $\rho_{A,B} (H)$ to stress the fact that every out-set (relative to $A$ and $B$) has at most $r$ vertices, so that implicitly by writing $\rho_{A,B}^r (H)$ we obtain that $R(H,A,B) \leq r$. Note that $\mathrm{id}(H)=\rho_{\emptyset,\emptyset}^k(H)$, and so our goal is then to show that $\rho_{\emptyset,\emptyset}^k(H)$ is rational.

From now on we focus on given finite, disjoint sets $A, B \subset V(H)$. We will prove a few claims that will cover all possible cases.

\bigskip
\emph{Claim~1}: If $A$ is not independent, then $\rho_{A,B}^r(H)=0$.
\smallskip

The first claim is obvious, as there exists no subset containing all vertices of $A$ and which is independent.

\bigskip
\emph{Claim~2}: If $A$ is independent and there is an infinite family of disjoint out-sets, then $\rho_{A,B}^r(H)=0$.
\smallskip

To prove the second claim, let $\mathcal{O}_1,\mathcal{O}_2,\ldots, \mathcal{O}_t \subseteq V(H)$ be a finite family of $t$ disjoint out-sets corresponding to the hyperedges $S_1,S_2,\ldots, S_t$ of $H$ which have empty intersection with $B$, all of sizes at most $r$. Since from any set $\mathcal{O}_i$ ($i \in \{1, 2, \ldots, t\}$) at least one subset of it (namely the whole set $\mathcal{O}_i$) will produce a hyperedge when merged with $A$, we have that
\begin{eqnarray*}
\rho_{A,B}^r(H) &\leq & 2^{-(|A|+|B|)} \prod_{i=1}^{t} \left(1- \frac{1}{2^{|\mathcal{O}_i|}}\right)  \\
                &\leq & 2^{-(|A|+|B|)} \left(1-\frac{1}{2^r} \right)^{t}.
\end{eqnarray*}

As $\lim_{t \rightarrow \infty} \left(1-\frac{1}{2^r}\right)^{t}=0$ the proof of the claim follows.

\bigskip
\emph{Claim~3}: If $A$ is independent, then $\rho_{A,B}^0(G)=2^{-(|A|+|B|)}$.
\smallskip

The third claim is obvious. Since the assumption is that each out-set has size at most $0$ there exists no out-set. This implies that each set containing $A$ that has empty intersection with $B$ is independent.

\bigskip
\emph{Claim~4}: Suppose that $A$ is independent and there is no infinite family of disjoint out-sets. If $\mathcal{O}_1, \mathcal{O}_2,\ldots,\mathcal{O}_s$ is a maximal family of disjoint out-sets (finite by assumption, perhaps empty), then for every $r>0$ we have that
$$
\rho_{A,B}^r(H) = \sum_{C \in 2^W} \rho_{A \cup C, B \cup (W \setminus C)}^{r-1}(H),
$$
where $W=\bigcup_{i=1}^s \mathcal{O}_i$.\\
(Note that, since $|W| \le sr < \infty$, the sum contains only finitely many terms.)
\smallskip

In order to prove the fourth claim, we apply~(\ref{eq:incl_excl2}). Since it is assumed that each out-set has at most $r$ vertices we can modify~(\ref{eq:incl_excl2}) as follows.
$$
\rho_{A,B}^r(H) = \sum_{C \in 2^W} \rho_{A \cup C, B \cup (W \setminus C)}^{r}(H).
$$
The crucial observation is that in our current setting $r$ on the right hand side can be (in fact, must be) replaced by $r-1$.  Note that out-sets which are hyperedges that are disjoint with $B$ must have non-empty intersection with some member of the family $\{\mathcal{O}_i: 1\le i \le s\}$ (since the family is maximal). Therefore, for any partition of $W$ into $C$ and $W \setminus C$, every out-set with respect to $A \cup C$ and $B \cup (W \setminus C)$ has at most $r-1$ vertices.  The claim follows.

\bigskip

With the claims in hand, we complete the proof of the theorem. Beginning with the term $\rho_{\emptyset,\emptyset}^k(H)$, we first fix a maximal family of $s$ disjoint out-sets. It might happen that $s$ is infinite, but in that case $\rho_{\emptyset,\emptyset}^k(H)=0$ by Claim~2.

Now assume that $s$ is finite. Since $A=\emptyset$ and $B=\emptyset$, this is just a collection of disjoint hyperedges $\mathcal{O}_1, \mathcal{O}_2,\ldots,\mathcal{O}_s$, with $W=\bigcup_{i=1}^s \mathcal{O}_i$.  We may apply Claim~4 to derive that
$$\rho_{\emptyset,\emptyset}^k(H) = \sum_{C \in 2^W} \rho_{C, W \setminus C}^{k-1}(H).$$
We will apply recursively one of the claims that corresponds to the terms in the sum. At each step, the term is either a rational number (Claims~1, 2, or 3), or a sum of finitely many terms which we can deal with independently (Claim 4). The recursive process ends after $k$ steps. We therefore have that $\mathrm{id}(H)=\rho_{\emptyset,\emptyset}^k(H)$ is rational.
\end{proof}

We now have the following corollary which is not obvious a priori.

\begin{corollary}\label{jjj}
If $\mathcal{F}$ is a finite set of finite graphs, then the $\mathcal{F}$-free density of a graph is rational.
\end{corollary}

Note that the proof of Theorem~\ref{rat} actually proves that the independence density of a hypergraph is a sum of dyadic rationals. However, dyadic rationals form a dense subset of the real numbers. The closure of the set of graph independence densities is rational, as proved in \cite{bbkp}. With Theorem~\ref{rat} in hand, we provide a short alternative proof of this result. It remains open, however, to determine if the closure of $\mathbf{H}_k$ is a subset of the rationals for $k>2$.

\begin{theorem}\label{closure_k=2}
The closure of $\mathbf{H}_2$ is a subset of $\mathbb{Q} \cap [0,1]$.
\end{theorem}
\begin{proof}
Fix $q \in (0,1) \setminus \mathbb{Q}$. We will use the argument used in the proof of Theorem~\ref{rat} to deduce that $q$ is not in the closure of $\mathbf{H}_2$.

Let $H = (V,E) \in \mathcal{H}_2$ be any graph. If
$$
\mu(H) \ge M = M(q) = \left\lceil \frac {\log (q/2)}{\log(3/4)} \right\rceil,
$$
then it follows from Theorem~\ref{bounds} that $\mathrm{id}(H) \le (3/4)^{\mu(H)} \le q/2$. On the other hand, if $\mu(H) < M$, then we may apply Claim~4 in the proof of Theorem~\ref{rat} to derive that there exists $W \subseteq V$ with $|W| \le 2(M-1)$ such that
\begin{equation}\label{uuu}
\mathrm{id}(H) = \rho_{\emptyset,\emptyset}^2(H) = \sum_{C \in 2^W} \rho_{C, W \setminus C}^{1}(H).
\end{equation}
The sum on the right hand side of~(\ref{uuu}) has at most $2^{2M-2}$ terms, each term is either equal to 0 (when the corresponding set $C$ is not independent or the partition $(C, W \setminus C)$ yields an infinite family of disjoint outsets) or of the form $1/2^i$ for some $i \in \mathbb{N}$ (when the partition $(C, W \setminus C)$ yields a finite family of outsets, with each outset consisting of one vertex).
Combining the two cases together we obtain that $H$ must either have $\mathrm{id}(H) \le q/2$ or have at most $2^{2M-2}$ ones in the binary expansion of $\mathrm{id}(H)$. This implies that there exists $\varepsilon = \varepsilon(q)>0$ such that there is no graph with independence density between $q-\varepsilon$ and $q+\varepsilon$. \end{proof}

\section{Irrationality in the case of unbounded hyperedges}\label{unbs}

One major difference between graphs and hypergraphs is that the latter can have variable hyperedge sizes. For infinite hypergraphs, these sizes can, of course, be unbounded. In this section we allow hypergraphs $H$ with finite hyperedges with unbounded (finite) cardinality (including the case of singletons as hyperedges).
We begin by providing (in sharp contrast to Corollary~\ref{jjj} where the sizes of hyperedges are bounded) an explicit example of a countable hypergraph with irrational independence density.

Let $\hat{H} \in \mathcal{H}$ be the hypergraph with vertices $\{x_i: i \in \mathbb{N}^+\}$ and hyperedges $$\{\{x_1\} , \{x_2,x_3\}, \{x_4,x_5,x_6\},\{x_7,x_8,x_9,x_{10}\}, \ldots \}.$$ In particular, if we consider the graph $\hat{G}$ which is the disjoint union of the cliques of order $1$ and $2$, and cycles of all lengths, then $\mathrm{id}(\hat{H})$ equals the acyclic density of $\hat{G}$.

\begin{theorem}\label{iii}
The following holds.
\begin{enumerate}
\item The independence density of $\hat{H}$ is $S= \prod_{k \geq 1} \left(1-\frac{1}{2^k}\right)$.
\item The number $S$ is irrational.
\end{enumerate}
\end{theorem}

Before we prove the theorem we make some observations.  As $\log (1-x) = - \sum_{n \geq 1} x^n/n$ for $-1 \leq x < 1$, we derive that
$$
S = \exp \left( \sum_{k\geq 1} \log \left(1-\df{1}{2^k}\right) \right) = \exp \left( - \sum_{k \geq 1} \sum_{n \geq 1} \frac{2^{-kn}}{n}\right).
$$
Since $\sum_{k \geq 1} \sum_{n \geq 1} \frac{2^{-kn}}{n} \le \sum_{k \geq 1} 2^{1-k} \le 4$, it is evident that $-\log S$ is bounded from above (in fact, it is approximately 1.242062) and so $S$ is bounded away from zero (and is roughly $0.288788$). Note that the fact that $\mathrm{id}(\hat{H})>0$ is in stark contrast to the result of Theorem~\ref{bounds}, since $\hat{H}$ has infinite matching number.

After interchanging sums we obtain another useful equality:
$$
S = \exp \left (- \sum_{n \geq 1} \frac{1}{n} \sum_{k \geq 1} (2^{-n})^k \right) = \exp \left( - \sum_{n \ge 1}  \frac {1}{n (2^n - 1)} \right).
$$
We need also the following theorem, which is often referred to as \emph{Euler's Pentagonal Number Theorem}.

\begin{theorem}[\cite{euler}]\label{euler}
For $|q|<1$ we have that
$$\prod_{k\geq 1} (1-q^k) = \sum_{k=-\infty}^\infty (-1)^kq^{k(3k+1)/2}.$$
\end{theorem}

Now, we are ready to come back to the proof of Theorem~\ref{iii}.

\begin{proof}[Proof of Theorem~\ref{iii}]
For item (1), note that hyperedges of $\hat{H}$ form an infinite family of disjoint hyperedges; there is exactly one hyperedge of cardinality $k$ for every $k \ge 1$. Consider a chain $\mathcal{C}=(H_n: n\ge 0)$ whose limit is $\hat{H},$ in which $H_n$ is a subhypergraph of $\hat{H}$ induced by the set of edges of cardinality at most $n$. It follows from Lemma~\ref{lem}~(2) (applied $n-1$ times) that $\mathrm{id}(H_n) =  \prod_{k = 1}^n \left(1-\frac{1}{2^k}\right)$ and the result holds after taking the limit as $n \to \infty$.

For item (2), note that $S$ is the sum at $q=1/2$ in Theorem~\ref{euler}.  For a contradiction, suppose that $S$ is rational.  Let $S=a/b$ with $\gcd(a,b)=1$.  For large enough $N$ the tail of the series is bounded as follows:
\begin{align*}
\left|\sum_{|k|\geq N} (-1)^k 2^{(-3k^2-k)/2}\right| &\leq \sum_{k\geq 0} \left( 2^{-(3(k+N)^2-k-N))/2} + 2^{-(3(k+N)^2+k+N)/2} \right) \\
& = 2^{-(3N^2-N)/2}\sum_{k\geq 0} 2^{-(3k^2+6kN-k)/2}(1 + 2^{-k-N}) \\
& < 2^{-N(3N-1)/2+1},
\end{align*}
where the last inequality follows by bounding the final sum by $2$. Indeed, for large enough $N$
$$
\sum_{k\geq 0} 2^{-(3k^2+6kN-k)/2}(1 + 2^{-k-N}) \le 1.5 \sum_{k\geq 0} 2^{-(3k^2+6kN-k)/2} \le 1.5 \sum_{k\geq 0} 2^{-3kN} < 2.
$$
On the other hand, the partial sum can be calculated as follows:
$$\sum_{-N+1 \leq k \leq N-1} (-1)^k 2^{-k(3k+1)/2} = \df{x}{2^{(N-1)(3N-2)/2}} = \df{x}{2^{(3N^2-5N+2)/2}}$$
for some integer $x$, since the least common multiplier of the denominators is simply the highest power of $2$ in the partial sum.  Thus, we obtain the string of inequalities
$$2^{-N(3N-1)/2+1} > \left|S-\df{x}{2^{(3N^2-5N+2)/2}}\right| \geq \df{1}{b2^{(3N^2-5N+2)/2}},$$
where the last inequality is true, provided that $S$ is not equal to the partial sum. Now, observe that if $|S-x/2^{(3N^2-5N+2)/2}|$ were equal to 0 for some particular $N$, then for $N+1$ the inequality would hold, as the summands corresponding to $k=N+1$ have nonzero contribution of order $2^{(-3(N+1)^2+(N+1))/2}$. Therefore, we derive the following string of inequalities for an infinite sequence of large enough $N$:
$$
\df{1}{b} \leq 2^{-N(3N-1)/2+1}2^{(3N^2-5N+2)/2} = 2^{-2N+2},
$$
contradicting that $b$ is fixed.
\end{proof}

Define
$$
\mathbf{H}_{unb}=\{ x\in [0,1]:\mbox{ for some countable }H, \mbox{ }\mathrm{id}(H)=x \}.
$$
As the next theorem demonstrates, the set of densities when hyperedges are unbounded can be any real number in $[0,1]$. This is in contrast to Theorem~\ref{rat} on the rationality of values in $\mathbf{H}_{k}$. Further, the set $\mathbf{H}_k$ contains gaps consisting of intervals such as $(1-1/2^k, 1).$

\begin{theorem}\label{unb}
$\mathbf{H}_{unb} = [0,1].$
\end{theorem}
\begin{proof}
Fix any $r \in [0,1]$, and write $r$ in binary expansion as $r=0.r_1r_2\ldots $, where $r_i \in \{0,1 \}$ for all $i\ge 1.$ It is possible that there are two such representations. In such a case, we will consider the representation with $r_i = 1$ for all $i \ge i_r$ for some $i_r \ge 1$. (For example, $r=7/8$ will be represented as $0.11011111\ldots$, not $0.11100000\ldots$.)

First, we construct a hypergraph $H_{(r)}$ which is a function of $r$. After that, we will show that its density is equal to $r$ by considering a suitable chain whose densities converge to $r.$

The vertex set of $H_{(r)}$ is $\mathbb{N}^+$. For $i \in \mathbb{N}^+$, let
$$
F_i = \{ j < i \text{ such that } r_j = 1 \}.
$$
That is, $F_i$ keeps track of the digits $j < i$ which are equal to $1$. Informally, the $i$th digit in the binary representation of $r$ (that is, $r_i$) corresponds to the vertex $i$ of $H_{(r)}$. For every $i \ge 1$ such that $r_i = 0$ we introduce a hyperedge $F_i \cup \{i\}$. This defines the hyperedge set of $H_{(r)}$. (For example, $H_{(7/8)}$ consists of one hyperedge only, $\{1,2,3\}$, whereas $H_{(1/\pi)}$ has an infinite number of hyperedges.)

For $n \in \mathbb{N}^+$, let $H_n$ be the hypergraph induced by the set of vertices $\{ 1, 2, \ldots, n\}$; $H_0$ is the empty graph. Since $(H_n:n\in \mathbb{N})$ forms a chain such  that $H_{(r)}=\lim_{n \rightarrow \infty}H_n$), we have that
$$
\mathrm{id}(H_{(r)})=\lim_{n \rightarrow \infty} \mathrm{id}(H_n).
$$
We will prove (by induction) that $\mathrm{id}(H_n)$ is the truncation of $r$ at the $n$th position rounded up to the nearest rational.

\smallskip
\noindent \emph{Claim}: For all integers $n\ge 1,$ we have that
$$
\mathrm{id}(H_n) = 0.r_1r_2\ldots r_n 111\ldots \ge r.
$$

The claim will finish the proof of the theorem, since the limit of the sequence $(\mathrm{id}(H_n): n \in \mathbb{N})$ tends to $r$.

Since $\mathrm{id}(H_0) = 1 = 0.111\ldots$, the base case holds. Suppose that the claim holds for all non-negative integers smaller than some $i \in \mathbb{N}$. If $r_i =1,$ then the hyperedge set of $H_{i-1}$ is exactly the same as the hyperedge set of $H_{i}$. The density remains the same and the claim holds in this case.
If $r_i =0,$ then exactly one more hyperedge is added to $H_{i-1}$ to form $H_{i}$; namely, $F_i \cup \{i\}$. Every independent set of $H_{i-1}$, except for the set $F_i$, can be extended to form an independent set of $H_i$ by adding vertex $i.$ Indeed, first note that $F_i$ itself is independent in $H_{(r)}$ (and so in $H_{i-1}$ as well) since each hyperedge has the property that the largest vertex corresponds to zero but all vertices in $F_i$ correspond to one. Moreover, any independent set in $H_{i-1}$ not containing all of $F_i$ remains independent after adding vertex $i$. Finally, if there were a proper superset $S$ of $F_i$ independent in $H_{i-1}$, then $S$ would contain a vertex $j<i$ corresponding to $0$. But then the set $F_j \cup \{j\}$ is already a hyperedge in $H_{i-1}$, contradicting the fact that $S$ is independent.  Thus, we lose exactly one possible extension, and so we have that
$$
\mathrm{id}(H_i) = \mathrm{id}(H_{i-1}) - 1/2^i,
$$
which is what we claimed to be the case when $r_i =0$.
\end{proof}

\section{Extensions to Independence Polynomials}\label{ext}

The {\em independence polynomial} of a hypergraph $H$ of order $n$ is the generating polynomial for the number of independent sets of each cardinality; that is, if $H$ has $i_{k}$ independent sets of size $k$, then the \emph{independence polynomial of }$H$ is defined by
\[ i(H,x) = \sum_{k \geq 0} i_{k}x^{k}.\]
Here we take $x\ge 0$ (although in other contexts, negative values are allowed). A trivial observation is that the number of independent sets of $H$ is equal to $i(H,1)$, so that the independence density of $H$ is given by
\[ \frac{i(H,1)}{2^n}.\]
For more on independence polynomials, see the survey \cite{lm}.

A natural extension of the definition of independence densities is the following. Let ${\mathcal C}$ be a chain of finite hypergraphs $H_{m}$ of orders $n_{m}$ with limit $H.$ Then for $x \geq 0$, we define the \emph{independence density of $H$ at $x$ with respect to the chain} ${\mathcal C}$ by
\[ \mathrm{id}_{\mathcal C}(H,x) = \lim_{m \rightarrow \infty} \frac{i(H_{m},x)}{2^{n_{m}}}\]
provided the limit exists (note that we allow the limit to be infinite). We note that a similar (but distinct) approach has been used towards limits of chromatic polynomials, especially the study of thermodynamic limits in the Potts model in statistical mechanics; see~\cite{psg,s}.

As proved in Section~\ref{inf}, $\mathrm{id}_{\mathcal C}(H,1) = \mathrm{id}(H)$ exists for any countable graph $H$, and is independent of the chain ${\mathcal C}$. Do analogous results hold for $\mathrm{id}_{\mathcal C}(H,x)$ for other values of $x$? The next lemma shows that for $ x \leq 1$, the answer is yes.

\begin{lemma}\label{lem:lem9}
Let $\mathcal C $ be chain of finite hypergraphs $H_{m}$ of orders $n_{m}$ with limit $H$. Then for $x \leq 1$,
\[ \mathrm{id}_{\mathcal C}(H,x) = \left\{ \begin{array}{ll} 0 & \mbox{if } x < 1\\
                                                         \mathrm{id}(H) & \mbox{if } x = 1. \end{array} \right. \]
\end{lemma}
\begin{proof}
We focus on the case when $x<1$ as the remaining case is immediate.
Observe that if $F$ is a spanning subhypergraph of a finite hypergraph $H$, then for $x \geq 0$, $i(H,x) \leq i(F,x)$ as every independent set of $H$ is an independent set of $F$. In particular, for a hypergraph $H$ of order $n$, $i(H,x)$ is at most the independence density polynomial for the hypergraph of order $n$ with no hyperedges, which is in turn equal to $(1+x)^n$. Hence, by dividing by $2^{n}$ and taking the limit, we see that for $x \in [0,1)$,
\[ 0 < \frac{i(H_{m},x)}{2^{n_{m}}} \leq \left( \frac{1+x}{2} \right)^{n_{m}}\]
and the right-most term tends to 0. By taking the limit as $m$ tends to $\infty$, we conclude that $\mathrm{id}_{\mathcal C}(H,x) = 0$.
\end{proof}

Note that for some countable hypergraphs, the independence density at $x$ is always 0. For example, as $i(K_{n},x)/2^{n} = (1+nx)/2^n$, it follows that the independence density at $x$ of the complete infinite graph $K_{\omega}$ is 0 for all $x \geq 0$ (and hence, independent of the chain). We can extend this by the following result which shows that if the independence numbers are not too large (as a function of the order of the hypergraph), then the independence density at $x$ is equal to $0$ for all $x \geq 0$.

\begin{theorem}\label{nnn}
Suppose that for each $n$, $H_{n}$ is a hypergraph of order $n$ with independence number $\beta_{n}$. Let $\omega=\omega(n)$ be any function tending to $\infty$ with $n$. If $\beta_{n} \leq n/\omega$, then for every $x > 0$ we have that
\[ \frac{i(H_{n},x)}{2^{n}} =o(1).\]
\end{theorem}
\begin{proof}
The result holds for $x < 1$ by Lemma~\ref{lem:lem9}. Fix $x\ge 1.$
We use the well known bound that for $1 \leq k \leq n$,
\[ {{n} \choose {k}} \leq \left( \frac{en}{k} \right)^{k}.\]
For $n$ sufficiently large so that $\beta_{n} < n/2$, we have that
\begin{eqnarray*}
i(H_{n},x)    & \leq & \sum_{k = 0}^{\beta_{n}}{{n} \choose {k}}x^{k}\\
              & \leq & (\beta_{n}+1){{n} \choose {\beta_{n}}}x^{\beta_{n}}\\
              & \leq & (\beta_{n}+1)\left( \frac{enx}{\beta_{n}} \right)^{\beta_{n}}.
\end{eqnarray*}
Hence,  we have that
\begin{eqnarray*}
 \frac{i(H_{n},x)}{2^n} &\leq & (\beta_{n}+1)\left( \frac{enx}{\beta_{n}} \right)^{\beta_{n}}/2^{n}. \\
 &=& \exp \Big(\log (\beta_{n} + 1) + \beta_{n} (1 +  \log n + \log x - \log \beta_{n}) - n\log 2 \Big)\\
 & \leq & \exp\Big(\log n + \frac{n}{\omega}(1+\log x +  \log \omega) - n\log 2 \Big)\\
 &=& \exp\Big( - n (\log 2 - o(1)) \Big) \\
 &=& o(1), %\qedhere
\end{eqnarray*}
and the proof is finished.
\end{proof}

We point out that the result of Theorem~\ref{nnn} is new even for independence densities of countable hypergraphs (that is, for $x = 1$), and provides a sufficient condition for independence densities to be 0 depending only on the independence numbers of hypergraphs in a chain.

The following theorem provides the limiting behaviour of $\mathrm{id}_{\mathcal C}(H, x)$ as $x$ tends to infinity.

\begin{theorem}\label{lim}
For a countable hypergraph which is the limit of the chain $\mathcal{C}=(H_m:m\in \mathbb{N})$, we have that
$$ \lim_{x\rightarrow \infty} \mathrm{id}_{\mathcal C}(H,x) \in \{ 0, \infty\}.$$
\end{theorem}

\begin{proof} We note first that $\mathrm{id}_{C}(H,x)$ is a non-decreasing function of $x,$ and so its limit as $x$ tends to $\infty$ either exists (and is non-negative) or is $\infty$. Let $n_m$ be the order of $H_m$. Suppose for a contradiction that $\lim_{x\rightarrow \infty} \mathrm{id}_{\mathcal C}(H,x) = z \in (0,\infty)$. Hence, there is an $x_0$ such that for all $x \ge x_0$ we have that
\begin{equation}\label{yy}
 z/2 < \mathrm{id}_{\mathcal C}(H,x) = \lim_{m \rightarrow \infty} \frac{i(H_{m},x)}{2^{n_{m}}} < 2z.
\end{equation}
We derive that
\begin{eqnarray*}
\mathrm{id}_{\mathcal C}(H,4 x_0) &=& \lim_{m \rightarrow \infty} \frac{i(H_{m},4 x_0)}{2^{n_{m}}} \\
&=& \lim_{m \rightarrow \infty} \frac {\sum_{k \geq 0} i_{k}(4 x_0)^{k}}{2^{n_{m}}} \\
&=& \lim_{m \rightarrow \infty} \frac {\sum_{k \geq 1} i_{k}( 4x_0)^{k}}{2^{n_{m}}} \\
&\ge& 4 \lim_{m \rightarrow \infty} \frac {\sum_{k \geq 1} i_{k} x_0^{k}}{2^{n_{m}}} \\
&=& 4 \lim_{m \rightarrow \infty} \frac {\sum_{k \geq 0} i_{k} x_0^{k}}{2^{n_{m}}} \\
&=& 4 \ \mathrm{id}_{\mathcal C}(H, x_0) > 2 z,
\end{eqnarray*}
which contradicts (\ref{yy}).
\end{proof}

\medskip

We conclude with some examples (focusing on graphs only) to show that, contrary to the situation for $x=1$, for $x>1$ the independence density at $x$ may depend on the chain, and may give non-dyadic rationals. The infinite path or \emph{ray} $P$ (either one- or two-way) is the limit of a chain ${\mathcal C}$ of paths $P_{n}$ of order $n$ (there are other chains whose limits are the infinite path, but we focus on this particular chain). The independence polynomials of paths $P_{n}$ satisfy the recurrence
\[ i(P_{n},x) = i(P_{n-1},x) + x \; i(P_{n-2},x),\]
with initial conditions $i(P_{1},x) = 1+x$ and $i(P_{2},x) = 1+2x$. In \cite{jichrjn} the recurrence was solved to derive
\begin{eqnarray*} i(P_{n},x) & = & \frac{\sqrt{1+4x}+(1+2x)}{2\sqrt{1+4x}}\left( \frac{1+\sqrt{1+4x}}{2} \right)^{n} + \\
 & & \frac{\sqrt{1+4x}-(1+2x)}{2\sqrt{1+4x}}\left( \frac{1-\sqrt{1+4x}}{2} \right)^{n},
\end{eqnarray*}
and thus,
\begin{eqnarray*} \frac{i(P_{n},x)}{2^n} & = & \frac{\sqrt{1+4x}+(1+2x)}{2\sqrt{1+4x}}\left( \frac{1+\sqrt{1+4x}}{4} \right)^{n} + \\
 & & \frac{\sqrt{1+4x}-(1+2x)}{2\sqrt{1+4x}}\left( \frac{1-\sqrt{1+4x}}{4} \right)^{n}.
\end{eqnarray*}
(We remark that by setting $x = 1$, we see that $i(P_{n},1)$ satisfies the same recurrence as the Fibonacci sequence $f_{n}$, though the initial terms are $f_{3}$ and $f_{4}$. The $n$th-term of the Fibonacci sequence is well known to be given by
$$f_{n} = \frac{1}{\sqrt{5}} \left( \left( \frac{1+\sqrt{5}}{2} \right)^{n} - \left( \frac{1-\sqrt{5}}{2} \right)^{n} \right),
$$
which, after computation, coincides with the formula given above for $i(P_{n-2},1)$.)

For $x \geq 0$, clearly the absolute value of $\frac{1+\sqrt{1+4x}}{4}$ dominates that of $\frac{1-\sqrt{1+4x}}{4}$. Now $\frac{1+\sqrt{1+4x}}{4}$ is an increasing function of $x$ and is equal to $1$ when $x = 2$. It follows that
\[ \mathrm{id}_{\mathcal C}(P,x) = \left\{ \begin{array}{ll}     0 & \mbox{if } 0 \leq x < 2\\
                                                          4/3 & \mbox{if } x = 2 \\
                                                       \infty & \mbox{if } x > 2.
\end{array} \right. \]

The example of the graph $K_{\omega} \cup \overline{K_{\omega}}$ is even more enlightening. It is the limit of chains of graphs $(H_n:n\in \mathbb{N})$ of the form $H_n = K_{a_n} \cup \overline{K_{b_n}}$, where $a_{n}$ and $b_{n}$ both tend to $\infty$ as $n \rightarrow \infty$. Now, for any $a, b \in \mathbb{N}$ and $x \ge 0$ we have that
$$
\frac{i(K_{a} \cup \overline{K_{b}},x)}{2^{a+b}} = \frac{(1+ax)(1+x)^{b}}{2^{a+b}} = \left( \frac{1+x}{2} \right)^{b} \frac{1+ax}{2^{a}}.
$$
It follows that for the chain ${\mathcal C}$ where $a_{n} = b_{n} = n$, we see that
\[ \frac{i(K_{a_n} \cup \overline{K_{b_n}},x)}{2^{a_n+b_n}} = \left( \frac{1+x}{4}\right)^{n}(1+nx) \]
which tends to 0 if $x < 3$ and to infinity if $x \geq 3$, so we have an example of a chain where the jump from 0 to infinity occurs with no intermediate point where the value is positive.

But we can do much more with the graph $K_{\omega} \cup \overline{K_{\omega}}$. Given a hypergraph $H$ with chain $\mathcal{C}$ define a \emph{jumping point} $x'\ge 1$ to be a real number such that $\mathrm{id}_{\mathcal C}(H,x')$ is finite but $\mathrm{id}_{\mathcal C}(H,x)=\infty$ if $x>x'.$ The next theorem shows that any real number not excluded by Lemma~\ref{lem:lem9} can be a jumping point.

\begin{theorem}\label{thm:jumping}
For every real number $r\in (1,\infty)$, there is a chain $\mathcal{C}_r$ in $K_{\omega} \cup \overline{K_{\omega}}$ for which $r$ is a jumping point. Moreover, for every chain in $\overline{K_{\omega}}$, $1$ is a jumping point.
\end{theorem}

\begin{proof}
Fix $r \in (1,\infty)$, and choose $C=C(r)>0$ such that $2^{C+1}=1+r$. Using the notation as in the example before the statement of the theorem, consider the chain $\mathcal{C}_r$ defined by $b_n = n$ and $a_n = \lfloor Cn \rfloor$. For $x < r$ we have that
\begin{eqnarray*}
\mathrm{id}_{{\mathcal C}_r}(K_{\omega} \cup \overline{K_{\omega}},x) &= &\lim_{n \rightarrow \infty} \frac {(1+\lfloor Cn \rfloor x)(1+x)^n}{2^{\lfloor Cn \rfloor + n}} \\
&\le& \lim_{n \rightarrow \infty}\frac {(1+Cnx)(1+x)^n}{2^{Cn-1 + n}} \\
&= &\lim_{n \rightarrow \infty} 2(1+Cnx) \left( \frac {1+x}{1+r} \right)^n = 0.
\end{eqnarray*}
On the other hand, for $x \ge r$ we have that
\begin{eqnarray*}
\mathrm{id}_{{\mathcal C}_r}(K_{\omega} \cup \overline{K_{\omega}},x) &= &\lim_{n \rightarrow \infty} \frac {(1+\lfloor Cn \rfloor x)(1+x)^n}{2^{\lfloor Cn \rfloor + n}}\\
&\ge &\lim_{n \rightarrow \infty} \frac {(Cnx)(1+x)^n}{2^{Cn + n}} \\
&= &\lim_{n \rightarrow \infty} (Cnx) \left( \frac {1+x}{1+r} \right)^n \\
&\ge & \lim_{n \rightarrow \infty} (Cnx) = \infty.
\end{eqnarray*}

The second part is straightforward. Any chain $(H_n:n\in \mathbb{N})$ with limit $\overline{K_{\omega}}$ is of the form $H_n = \overline{K_{b_n}}$ and $b_n \to \infty$. We have that
\[ \frac{i(\overline{K_{b_n}},x)}{2^{b_n}} = \left( \frac{1+x}{2}\right)^{b_n}, \]
which tends to 0 if $x < 1$, to 1 if $x=1$, and to infinity if $x>1$.
\end{proof}
The graph $G = K_{\omega} \cup \overline{K_{\omega}}$ admits chains where, for a given $x>1$, $\mathrm{id}_{{\mathcal C}}(G,x)$ does not exist. As before, we consider chains of graphs $(G_n:n\in \mathbb{N})$ of the form $G_n = K_{a_n} \cup \overline{K_{b_n}}$, and use the notation as in the proof of Theorem~\ref{thm:jumping}. Fix $x>1$, let $C_1 = C((1+x)/2)$, $C_2 = C(2x)$, and consider chains ${\mathcal C_i} =  (G^i_{n}: n \in \mathbb{N})$, where $a^i_n=\lfloor C_i n \rfloor$ and $b^i_n=n$ (where $i = 1,2$). (Note that $C_2 > C_1$, since $C=C(x)$ is an increasing function of $x$.) The constants $C_1, C_2$ are chosen so that $\mathrm{id}_{\mathcal C_1}(G,x)=\infty$ and $\mathrm{id}_{\mathcal C_2}(G,x)=0$.

Now consider a chain ${\mathcal C_3}=( G^3_{n}: n \in \mathbb{N})$, where for even $n$ a graph from ${\mathcal C_1}$ is taken, and for odd $n$ we take a graph from ${\mathcal C_2}$. Since both $a^i_n$ and $b^i_n$ tend to infinity ($i=1,2$), it is always possible to select a graph that has the previous one as an induced subgraph. Then $\mathrm{id}_{\mathcal C_3}(G,x)$ does not exist, since its corresponding sequence of densities has a subsequence tending to infinity and a subsequence tending to zero.

%We mention that it is open to determine whether $\mathrm{id}_{{\mathcal C}}(H,x)$ is rational when it is finite and exists.

\section{Acknowledgements}

We thank O-Yeat Chan for helpful discussions.

\end{document}